\documentclass[oneside,10pt]{amsart}
\usepackage{amsmath}
\usepackage{amssymb}
\usepackage{amsthm}
\usepackage{float}
\usepackage{ulem}
\usepackage{xcolor}
\usepackage[utf8x]{inputenc}
\usepackage{comment}
\usepackage{framed}
\usepackage{caption}
\usepackage{graphicx} 

\newcommand{\abs}[1]{\left\vert#1\right\vert}

\newcommand{\eps}{\varepsilon}

\DeclareMathAlphabet{\mathbbb}{U}{bbold}{m}{n}

\newtheorem{thm}{Theorem}[section]
\newtheorem{cor}[thm]{Corollary}
\newtheorem{lem}[thm]{Lemma}
\newtheorem{prop}[thm]{Proposition}
\theoremstyle{definition}

\theoremstyle{remark}
\newtheorem{rem}[thm]{Remark}

\numberwithin{equation}{section}
\usepackage{graphicx}

\usepackage{enumerate}

\newcommand{\matrice}{\begin{pmatrix}}
\newcommand{\ok}{\end{pmatrix}}

\begin{document}
\title[]{Isoperimetric inequalities and sharp upper bounds for Aharonov-Bohm eigenvalues on surfaces}
\thanks{The first author acknowledges support of INDAM-GNAMPA. The second and third authors acknowledge support of INDAM-GNSAGA. The second author acknowledges support of the INDAM-GNSAGA 2025 project ``Analisi geometrica e teoria spettrale su varietà riemanniane ed hermitiane'' and of the the project ``Perturbation problems and asymptotics for elliptic differential equations: variational and potential theoretic methods" funded by the European Union – Next Generation EU and by MUR-PRIN-2022SENJZ3. The authors would like to thank the Isaac Newton Institute for Mathematical Sciences, Cambridge, for support and hospitality during the programme Geometric spectral theory and applications, where work on this paper was undertaken. This work was supported by EPSRC grant EP/Z000580/1.}

\author[Michetti]{Marco Michetti}
\address{Department of Mathematical Sciences, Chalmers University of Technology and the University of Gothenburg, SE-41296 Gothenburg, Sweden, e-mail: {\sf michetti@chalmers.se}.}
\author[Provenzano]{Luigi Provenzano}
\address{Dipartimento di Scienze di Base e Applicate per l'Ingegneria, Universit\`a di Roma ``La Sapienza'', Via Scarpa 12 - 00161 Roma, Italy, e-mail: {\sf luigi.provenzano@uniroma1.it}.}
\author[Savo]{Alessandro Savo}
\address{Dipartimento di Scienze di Base e Applicate per l'Ingegneria, Universit\`a di Roma ``La Sapienza'', Via Scarpa 12 - 00161 Roma, Italy, e-mail: {\sf alessandro.savo@uniroma1.it}.}

\begin{abstract}
We consider the first eigenvalue of the magnetic Laplacian with zero magnetic field on simply connected compact surfaces and we establish isoperimetric inequalities and upper bounds in terms of a bound on the gaussian curvature. As a corollary, we prove that among all simply connected spherical domains of fixed area, the first eigenvalue is maximal for a geodesic disk with  the pole of the magnetic potential at its center; also, for the sphere punctured at two points, the first eigenvalue is maximal when the punctures are antipodal.

\end{abstract}


\maketitle

\section{Introduction}
In this paper we consider the first eigenvalue, or {\it ground state energy}, of the magnetic Laplacian on simply connected domains of the sphere and, more generally, on compact, simply connected, Riemannian surfaces. 

\smallskip

Let $(M,g)$ be a compact Riemannian manifold (of any dimension, with or without boundary). A {\it magnetic potential} is a smooth real $1$-form $A$ on $M$. It defines a {\it magnetic differential} $d^A$ acting on smooth, complex-valued functions as follows:
$$
d^Au=du-iuA,
$$
and a {\it magnetic codifferential} $\delta^A$, acting on $1$-forms $\omega$ as $\delta^A(\omega)=\delta\omega+i\langle A,\omega\rangle_g$. Here $\delta$ is the usual codifferential associated with the metric $g$, which in $\mathbb R^n$ with the Euclidean metric is $-{\rm div}$. The {\it magnetic Laplacian} acts on smooth complex-valued functions as follows:
$$
\Delta_Au=\delta^Ad^Au.
$$
When $\partial M\ne\emptyset$, we consider magnetic Neumann boundary conditions $d^Au(N)=0$, which in terms of the dual vector field  $\nabla^Au:=\nabla u-iuA^{\sharp}$, read $\langle\nabla^Au,N\rangle=0$, where $N$ is the unit outer normal to $\partial M$. Here $A^{\sharp}$ denotes the vector field dual to the $1$-form $A$; the expression $\nabla^Au$ denotes the {\it magnetic gradient} of $u$. Then $\Delta_A$ defines a compact self-adjoint operator on $L^2(M,dv_g)$, hence it has a non-negative spectrum made of eigenvalues of finite multiplicity diverging to $+\infty$:
$$
0\leq\mu_1(M,g,A)\leq\mu_2(M,g,A)\leq\cdots\leq\mu_k(M,g,A)\leq\cdots\nearrow+\infty.
$$
As usual, the eigenvalues are characterized by the min-max principle:
$$
\mu_k(M, g, A)=\min_{{\substack{U\subset H^1_A(M)\\{\rm dim}U=k}}}\max_{0\ne u\in U}\frac{\int_{M}|d^Au|_g^2dv_g}{\int_{M}|u|^2dv_g},
$$
where $H^1_A(M)$ is the {\it magnetic Sobolev space}, namely, the space of complex-valued functions in $L^2(M,dv_g)$ with $d^Au$ in $L^2(M,dv_g)$. If $A$ is smooth, this space coincides with the usual Sobolev space $H^1(M,dv_g)$.
Here and in what follows we drop the dependence of the eigenvalue on the metric when it is clear from the context, and if $(M,g)$ is a Riemannian manifold, we write simply
$$
\mu_k(M,g,A)\doteq\mu_k(M,A).
$$

\smallskip

Analogously, one can define  the magnetic Dirichlet eigenvalues:
$$
0<\lambda_1(M,g,A)\leq\lambda_2(M,g,A)\leq\cdots\leq\lambda_k(M,g,A)\leq\cdots\nearrow+\infty
$$
which are characterized by
\begin{equation}\label{ev_dir}
\lambda_k(M, g, A)=\min_{{\substack{U\subset H^1_{0,A}(M)\\{\rm dim}U=k}}}\max_{0\ne u\in U}\frac{\int_{M}|d^Au|_g^2dv_g}{\int_{M}|u|^2dv_g},
\end{equation}
where $H^1_{0,A}(M)$ is the subspace of $H^1_A(M)$ of functions with zero boundary trace. Also in this case we set $\lambda_k(M,g,A)\doteq\lambda_k(M,A)$. The magnetic Dirichlet eigenvalues will appear in the proof of the main results.

\subsection{Gauge invariance.} A fundamental feature of the magnetic Laplacian is {\it gauge invariance}. Recall that the {\it flux} of a $1$-form $A$ around a closed loop $c$ in $M$ is defined as $\frac{1}{2\pi}\oint_cA$. We have the following fundamental theorem which characterizes potentials that can be ``{\it gauged-away}'' (see \cite{Sh}, see also \cite[Appendix 5]{CEIS}).
\begin{thm}\label{shi}

\begin{enumerate}[a)]
\item Assume that $A,A'$  are such that $A-A'$ is closed and $\frac{1}{2\pi}\oint_c(A-A')\in\mathbb Z$ for  any simple closed curve $c$ in $M$. Then $\mu_k(M,A)=\mu_k(M,A')$.
\item One has $\mu_1(M,A)=0$ if and only if $A$ is closed and 
$\frac{1}{2\pi}\oint_cA\in\mathbb Z$ for  any simple closed curve $c$ in $M$.
\end{enumerate}
\end{thm}
In particular, if the flux of $A$ around any closed curve in $M$ is an integer (in particular, if $A$ is exact) then the magnetic spectrum coincides with the usual Laplacian spectrum for Neumann boundary conditions, which we simply denote $\mu_k(M)$: that is 
$$
\mu_k(M,A)=\mu_k(M) \quad\text{for all $k$}.
$$
According to this numbering $\mu_1(M)=0$ so that the first {\it positive} Neumann eigenvalue is then $\mu_2(M)$.

We remark that, classically, the $2$-form $B\doteq dA$ is the {\it magnetic field} associated to $A$. We then remark that the magnetic field does not determine the spectrum; in particular, the magnetic field could be zero and yet the lowest eigenvalue (ground state energy) is positive; this happens when the flux of $A$ around some curve is not an integer.

\subsection{Aharonov-Bohm potentials}

The scope of the present paper is precisely to estimate the ground state energy when the magnetic field is zero on the regular part of singular potentials called {\it Aharonov-Bohm potentials}, which are the object of intensive work, both from the physical and the mathematical point of view. 

On a simply connected domain $\Omega$, an Aharonov-Bohm potential with pole at $x_0\in\Omega$ is a closed smooth $1$-form on $\Omega\setminus\{x_0\}$.  The circulation of $A$ around $\partial\Omega$ (and around any loop enclosing $x_0$) is a real number $\nu$ which is called the {\it flux} of $A$.

In $\mathbb R^2$, an example of Aharonov-Bohm potential with pole at the origin and flux $\nu$ is
$$
A^{(\nu)}_0=\nu\left(-\frac{y}{x^2+y^2}dx+\frac{x}{x^2+y^2}dy\right).
$$
 We consider magnetic Neumann boundary conditions, so that the spectrum is discrete (see e.g., \cite{CPS22}). By what we said above $\mu_1(\Omega,A^{(\nu)}_0)>0$ if and only if $\nu\notin\mathbb Z$.  

 We remark that any two Aharonov-Bohm potentials with pole at $x_0$ and flux $\nu$ differ by an exact form, hence give rise to the same spectrum. 

In this paper we are interested in shape optimization results for the first Aharonov-Bohm eigenvalue of compact, simply connected surfaces (with or without boundary). 

\subsection{State of the art}
\ \\
{\bf Punctured domains in $\mathbb R^2$.} The following isoperimetric inequality has been proved in \cite{CPS22}: for any bounded, smooth domain $\Omega\subset\mathbb R^2$ and any flux $\nu$ we have 
\begin{equation}\label{CPS}
\mu_1(\Omega,A^{(\nu)}_{x_0})\leq \mu_1(D,A^{(\nu)}_0)
\end{equation}
for all $x_0\in\Omega$.
Here $D$ is the disk centered at $0$ with the same area: $\abs{\Omega}=\abs{D}$. Equality holds if and only if $\Omega=D$ and $x_0=0$. In fact, the isoperimetric inequality \eqref{CPS} holds for any domain (not necessarily simply connected). The proof uses the classical argument of Weinberger for the second eigenvalue of the Neumann Laplacian \cite{weinberger}. Note that in polar coordinates $(r,\theta)$ around $0$, $A^{(\nu)}_0=\nu d\theta$. The result holds also for bounded, smooth domains of the hyperbolic plane $\mathbb H^2$, with the same proof.

\smallskip
\noindent{\bf Spherical domains.}  On the two-dimensional round sphere $\mathbb S^2$, fix a point $x_0$ and consider polar coordinates $(r,\theta)$ based at $x_0$, where $r$ is the geodesic distance from $x_0$. Then $A_{x_0}^{(\nu)}\doteq\nu d\theta$ defines a smooth $1$-form on $\mathbb S^2\setminus\{x_0,-x_0\}$ which is singular at the two poles: $x_0$ and $-x_0$, and defines a Aharonov-Bohm potential on $\mathbb S^2$. We note that the flux around any loop that separates $x_0$, $-x_0$ is equal to $\nu$. The following inequality of isoperimetric type is proved in  \cite{CPS22}.

\begin{thm}\label{cps} Fix $x_0\in\mathbb S^2$ and let $\Omega$ be a smooth bounded domain that contains $x_0$ but not $-x_0$. Let $D$ be the spherical cap centered at $x_0$ with the same area: $\abs{\Omega}=\abs{D}$. Then, the isoperimetric inequality 
\begin{equation}\label{sphereCPS}
\mu_1(\Omega,A^{(\nu)}_{x_0})\leq \mu_1(D,A^{(\nu)}_{x_0})
\end{equation}
holds for all $\nu\in\mathbb R$ under any of the following assumptions:

\begin{enumerate}[a)]
\item $\Omega$ is contained in the hemisphere centered at $x_0$;

\item $\Omega$ is simply connected with area $\abs{\Omega}\leq \frac 12\abs{\mathbb S^2}=2\pi$.
\end{enumerate}
If $\nu\notin\mathbb Z$, equality holds in a) and b) if and only if $\Omega$ is the spherical cap $D$ centered at $x_0$.
\end{thm}

We remark that under any of the assumptions a) and b), the isoperimetric inequality is known to hold for the lowest positive eigenvalue of the usual Neumann Laplacian, that is:
\begin{equation}\label{classical}
\mu_2(\Omega)\leq\mu_2(D);
\end{equation}

\noindent precisely,  it holds under assumption a) by an argument of Weinberger (see \cite{ash_beng}),
and it holds under assumption b) by conformal transplantation (see Bandle \cite{bandle2} and  Szeg\"o \cite{szego_1}). The proof of Theorem \ref{cps} is an adaptation of Weinberger's and
Szeg\"o' s arguments.  

We point out a recent growing interest in the eigenvalues of the Neumann Laplacian on spherical domains. In particular, we mention \cite{laugesen} where, in the inequality \eqref{classical} for simply connected domains, the area bound is relaxed; precisely  from $|\Omega|\leq 0.5|\mathbb S^2|$  to $|\Omega|\leq \frac{16}{17}|\mathbb S^2|$: in other words, \eqref{classical}  holds for simply connected domains with area at most $94\%$ that of the sphere.  In the recent paper \cite{bucur_sphere} the authors show that if the hypothesis on simply-connectedness is dropped, then spherical caps are not necessarily maximisers of the second Laplace eigenvalue among domains of fixed area. 

The approach we propose in this paper for the study of isoperimetric inequalities and eigenvalue upper bounds is different from those of Szeg\"o and Weinberger. In fact, among other results, we prove here that Theorem \ref{cps}, b) holds with no restrictions on $|\Omega|$. We remark that adapting the method to the case of the second Neumann eigenvalue of the standard Laplacian one can prove \eqref{classical} for all simply connected spherical domains, with no restrictions on the area. This is done in \cite{PS_sphere}.

\subsection{Aims of the paper}

Let us give an analytical description of an Aharonov-Bohm potential on a compact, simply  connected Riemannian surface $\Omega$ with non-empty boundary $\partial\Omega$ and pole $z_0$.  
Consider, on the punctured domain $\Omega\setminus\{z_0\}$ the smooth $1$-form
$$
A_{z_0}^{(\nu)}\doteq -2\pi\nu\star d\psi_{z_0}\,,
$$
where $\psi_{z_0}$ is the Green function with pole at $z_0$, which is defined by
\begin{equation}\label{green_boundary}
\begin{cases}
\Delta\psi_{z_0}=\delta_{z_0}\,, & {\rm in\ } \Omega\\
\psi_{z_0}=0\,, & {\rm on\ }\partial \Omega,
\end{cases}
\end{equation}
and $\star$ is the Hodge-star operator associated with the Riemannian metric. Note that $A_{z_0}^{(\nu)}$ is closed and has flux $\nu$ around $\partial\Omega$ (and around  any loop enclosing $z_0$).

\smallskip

On the other hand, if $\Sigma$  is a compact, simply connected Riemannian surface with empty boundary, an Aharonov-Bohm potential is determined by the flux $\nu$ and two poles $z_1,z_2$. We set, on $\Sigma\setminus\{z_1,z_2\}$:
$$
A_{z_1,z_2}^{(\nu)}\doteq -2\pi\nu\star d\psi_{z_1,z_2}\,,
$$
where $\psi_{z_1,z_2}$ is the Green function with poles at $z_1, z_2\in\Sigma$, namely
\begin{equation}\label{green_closed}
\Delta\psi_{z_1,z_2}=\delta_{z_1}-\delta_{z_2}\,,\ \ \ {\rm in\ } \Sigma.
\end{equation}
In this case, $\nu$ is the flux of $A^{(\nu)}_{z_1,z_2}$ around any loop separating the poles and $\delta_P$ denotes the Dirac delta measure at a point $P$.

\smallskip

As we already mentioned, in the case of a geodesic disk in $\mathbb S^2$ with pole $x_0$ in its center, $A^{(\nu)}_{x_0}$ takes the simpler form $A^{(\nu)}_{z_0}=\nu d\theta$, where $(r,\theta)$ are the standard polar coordinates centered at $x_0$ (with $r$ being the distance from $x_0$). This is also the form of the Aharonov-Bohm potential on the standard sphere $\mathbb S^2$ with antipodal poles $x_0$, $-x_0$.

\smallskip

In both cases the spectrum is well-behaved, and we are interested in isoperimetric inequalities and upper bounds for the first eigenvalue, which is positive as long as $\nu\notin\mathbb Z$.

\smallskip

The inequalities are obtained by employing the {\it method of prescribed level lines} which could be dated back to Rayleigh and Ritz (see e.g., \cite{polya_szego}), and by applying it to the Green function of $\Omega$ or $\Sigma$ as above. The same method, using test functions which are constant on the level lines of the {\it torsion} function, has been used in \cite{CLPS_reverse} to establish a reverse Faber-Krahn inequality for the first eigenvalue of the magnetic Laplacian with constant magnetic field on simply connected planar domains, in the weak field regime. 

\smallskip

First, we consider the isoperimetric problem and prove that among all simply connected surfaces with boundary, having gaussian curvature bounded above by $1$ and fixed area smaller than $|\mathbb S^2|=4\pi$, the first eigenvalue is maximal for a geodesic disk in $\mathbb S^2$ with the pole of the potential at its center. In particular, among all simply connected spherical domains of fixed area, the geodesic disk with pole at the center is maximal. Note that we have no restriction on the area, extending in this way the result of \cite{CPS22}, in particular Theorem \ref{cps}, b). 

 \smallskip

Another natural question is the following. Consider all simply connected, closed (i.e. compact with empty boundary) Riemannian surfaces of fixed area punctured at two points:  {\it does there exist a maximiser for the first Aharonov-Bohm  eigenvalue?} For the second eigenvalue of the usual Laplacian the answer is positive: it is the round sphere $\mathbb S^2$, and the result is due to Hersch \cite{hersch_sphere}, and can be conveniently stated in terms of the lowest normalized eigenvalue, which is invariant by homotheties:
$$
\sup_{\Sigma}\abs{\Sigma}\mu_2(\Sigma)=\abs{\mathbb S^2}\mu_2(\mathbb S^2)=8\pi.
$$

A bit surprisingly, in the magnetic case there is no maximiser for $\abs{\Sigma}\mu_1(\Sigma,A^{(\nu)}_{z_1,z_2})$; in fact,  if $\nu\notin\mathbb Z$:
$$
\sup_{\substack \Sigma,\,{z_1,z_2\in\Sigma}}\abs{\Sigma}\mu_1(\Sigma,A^{(\nu)}_{z_1,z_2})=+\infty,
$$
see Theorem \ref{nohersch}. The point is that when the flux is not an integer, the first eigenvalue depends on the positions of the poles. When the flux is an integer,  the position of the poles have no influence on the spectrum, by gauge invariance. In this paper we prove an upper bound which depends only on an upper bound on the gaussian curvature $K$, see Theorem \ref{main_closed}.  The dependence on $K$ is sharp.

 \smallskip If the surface is $\mathbb S^2$ punctured at two poles, a natural question is then: {\it what is the optimal position of the poles, for which the first eigenvalue is maximal?} We show in Corollary \ref{cor_2poles} that the two poles must be antipodal (in the appendix we compute explicitly the spectrum in that case).

 \smallskip
Even if the focus of the paper is on simply connected surfaces, with or without boundary, we study here also the case of Riemannian annuli (surfaces diffeomorphic to $\mathbb S^1\times [0,1]$). Then we are able to prove that the first Aharonov-Bohm eigenvalue is maximal for the unique flat cylinder in a given conformal class.

\subsection{Organization of the paper}The paper is organized as follows: in Section \ref{sec_main_results} we state the main results, namely Theorem \ref{main_boundary} (isoperimetric inequality) and \ref{main_closed} (sharp upper bound) and their consequences on spherical domains, namely Corollaries \ref{cor_boundary_spherical} and \ref{cor_2poles}. In Section \ref{sec_main_boundary} we prove Theorem \ref{main_boundary}. In Section \ref{sec_main_closed} we prove Theorem \ref{main_closed} and discuss its sharpness. We have included two appendices. In Appendix \ref{appendix} we compute explicitly the magnetic spectrum of the sphere with opposite poles and discuss other properties of eigenvalues on spherical domains. In Appendix \ref{appendixb} we discuss the case of Riemannian annuli.

\section{Main results} \label{sec_main_results}

 
\subsection{Isoperimetric inequality for simply connected surfaces with boundary} For a constant $K>0$, we let
$\mathcal M_K$ denote the family of compact simply connected Riemannian surfaces with smooth boundary and gaussian curvature bounded above by $K$. When $\abs{\Omega}\leq \abs{\mathbb S^2_K}$ the model domain for our eigenvalue problem will be a spherical cap in  $\mathbb S^2_K$ with same area (eventually, the whole sphere when $\abs{\Omega}= \abs{\mathbb S^2_K}$). We get in that case:
 
\begin{thm}\label{main_boundary} Let $\Omega\in \mathcal M_K$, fix a pole $z_0\in\Omega$ and let $A_{z_0}^{(\nu)}$ be an Aharonov-Bohm potential with arbitrary flux $\nu\in\mathbb R$ and pole $z_0$. Assume that 
$\abs{\Omega}\leq \abs{\mathbb S^2_K}$. Then
$$
\mu_1(\Omega,A_{z_0}^{(\nu)})\leq \mu_1(D^{\star},A_{x_0}^{(\nu)}),
$$
where $D^{\star}$ is the spherical cap in $\mathbb S^2_K$ centered at some $x_0\in\mathbb S^2_K$ with area $\abs{\Omega}$. Equality holds if and only if $\Omega=D^{\star}$ and $z_0=x_0$.
\end{thm}

The following is an immediate consequence:
\begin{cor}\label{cor_boundary_spherical} Let $\Omega$ be a simply connected domain in $\mathbb S^2$, and choose any $z_0\in\Omega$. Then, for any flux $\nu\in\mathbb R$ we have:
$$
\mu_1(\Omega,A_{z_0}^{(\nu)})\leq \mu_1(D^{\star},A_{x_0}^{(\nu)}),
$$
where $D^{\star}$ is the spherical cap in $\mathbb S^2$ centered at some $x_0\in\mathbb S^2$ with area $\abs{\Omega}$. Equality holds if and only if $\Omega=D^{\star}$ and $z_0=x_0$.
\end{cor}
This inequality was first proved with the restriction $|\Omega|\leq 2\pi$ in \cite{CPS22}.

When the domain $\Omega\in \mathcal M_K$ has area greater than that of $\mathbb S^2_K$ the model domain will be the whole sphere, in the following sense.

\begin{thm}\label{main_closed_2} Let $\Omega\in {\mathcal M}_K$, let $z_0\in\Sigma$ and let $A_{z_0}^{(\nu)}$ be an Aharonov-Bohm potential with  flux $\nu\in(0,\frac 12]$ and pole $z_0$. Assume that $\abs{\Omega}\geq \abs{\mathbb S^2_K}$. Then we have:
\begin{equation}\label{ineq_main_closed_2}
\begin{aligned}
\mu_1(\Sigma,A_{z_0}^{(\nu)})&\leq K\nu(\nu+1)\\
&=\mu_1(\mathbb S^2_K,A_{x_0}^{(\nu)})\\
\end{aligned}
\end{equation}
where $A_{x_0}^{(\nu)}$ is the Aharonov-Bohm potential on ${\mathbb S}^2_K$ with antipodal poles $x_0$, $-x_0$. Equality holds if and only if $\Omega=\mathbb S^2_K\setminus\{-x_0\}$ with pole $x_0$. If $\nu\in\mathbb R$, the inequality holds replacing $\nu$ by $\min_{n\in\mathbb Z}\abs{\nu-n}$. 
\end{thm}

\subsection{Upper bounds for simply connected surfaces} For a constant $K>0$, we let
$\overline{\mathcal M}_K$ denote the family of closed simply connected Riemannian surfaces with gaussian curvature bounded above by $K$. The model surface is here 
$\mathbb S^2_K$ with antipodal poles. Here is the relevant result. 

\begin{thm}\label{main_closed} Let $\Sigma\in \overline{\mathcal M}_K$, let $z_1, z_2\in\Sigma$ and let $A_{z_1,z_2}^{(\nu)}$ be an Aharonov-Bohm potential with flux $\nu$ and poles $z_1,z_2$. Then, for any flux $\nu\in (0,\frac 12]$ we have:
\begin{equation}\label{ineq_main_closed}
\begin{aligned}
\mu_1(\Sigma,A_{z_1,z_2}^{(\nu)})&\leq K\nu(\nu+1)\\
&=\mu_1(\mathbb S^2_K,A_{x_0}^{(\nu)})\\
\end{aligned}
\end{equation}
where $A_{x_0}^{(\nu)}$ is the Aharonov-Bohm potential on ${\mathbb S}^2_K$ with antipodal poles $x_0,-x_0$. Equality holds if and only if $\Sigma=\mathbb S^2_K$ and $z_1, z_2$ are antipodal. If $\nu\in\mathbb R$, the inequality holds replacing $\nu$ by $\min_{n\in\mathbb Z}\abs{\nu-n}$. 
\end{thm}

As an immediate consequence, we take $\Sigma=\mathbb S^2$ and observe the following corollary on the optimal placement of two poles on a round sphere. 

\begin{cor}\label{cor_2poles} The following inequality holds for any pair of poles $x_0,y\in\mathbb S^2$ and any flux $\nu\in (0,\frac 12]$:
$$
\begin{aligned}
\mu_1(\mathbb S^2, A^{(\nu)}_{x_0,y})&\leq \nu(\nu+1)\\
&=\mu_1(\mathbb S^2, A^{(\nu)}_{x_0})
\end{aligned}
$$
where $A_{x_0}^{(\nu)}$ is the Aharonov-Bohm potential on ${\mathbb S}^2_K$ with antipodal poles $x_0,-x_0$. Equality holds if and only if  $y=-x_0$.
\end{cor}

\subsection{Riemannian annuli}

 Finally, we briefly consider the case of a Riemannian annulus $\mathcal C$, namely, a compact surface diffeomorphic to $\mathbb S^1\times[0,1]$ endowed with a Riemannian metric. Let $A^{(\nu)}$ be a smooth closed $1$-form on $\Sigma$ with flux $\nu\in\mathbb R$ around any of the boundary circles. For example, in cylindrical coordinates $(\theta,z)$, one can take $A^{(\nu)}=\nu d\theta$. We consider the magnetic Laplacian $\Delta_{A^{(\nu)}}$. Then the Neumann spectrum is discrete, and the first eigenvalue is positive as long as $\nu\notin\mathbb Z$ and one has by the min-max principle
 $$
\mu_1(\mathcal C,A^{(\nu)})=\min_{0\ne u\in H^1(\mathcal C)}\frac{\int_{\mathcal C}|d^Au|^2}{\int_{\mathcal C}|u|^2}.
 $$
Since any two closed $1$-forms with the same flux on $\mathcal C$ differ by an exact $1$-form, we see by gauge invariance that the magnetic spectrum depends only on $\nu$. Hence we  set
$$
\mu_1(\mathcal C,\nu)\doteq\mu_1(\mathcal C,A^{(\nu)}).
$$
Now, any Riemannian annulus $\mathcal C$ is conformally equivalent to the  flat cylinder $\mathcal C_M\doteq \mathbb S^1\times[-M,M]$ for a unique $M>0$; the number $M$ is the so-called {\it conformal modulus} of $\mathcal C$. We refer e.g., to \cite{PSannuli} where the above study is carried out for the magnetic Steklov problem. Here is the relevant result:
 \begin{thm}\label{main_annulus}

$$
|\mathcal C|\mu_1(\mathcal C,\nu)\leq4\pi M\min_{n\in\mathbb Z}|\nu-n|^2
$$
where $M$ it the conformal modulus of $\mathcal C$. Equality holds if and only if $\mathcal C$ is homothetic to $\mathcal C_M$. 
 \end{thm}
Note that the upper bound depends on the minimum distance of $\nu$ from the set of integer numbers. In particular, we can take $\nu\in(0,\frac 12]$ so that the upper bound reads
$$
|\mathcal C|\mu_1(\mathcal C,\nu)\leq4\pi M\nu^2.
$$
With the same method as in \cite{CPS_conformal}, we have that the inequality is sharp asymptotically as $\nu\to 0$, because:
$$
\lim_{\nu\to 0}\frac{|\mathcal C|\mu_1(\mathcal C,\nu)}{\nu^2}=4\pi M.
$$
In other words, knowing the normalised eigenvalue $|\mathcal C|\mu_1(\mathcal C,\frac 1n)$ for all $n\in\mathbb N$, determines the conformal class of $\mathcal C$. In \cite{CPS_conformal} this phenomenon is studied for closed, orientable surfaces of arbitrary genus $g$: there it is proved that the knowledge of a suitable countable collection of magnetic ground states determines the volume and the conformal class of the surface.

\section{Proof of Theorem \ref{main_boundary}}\label{sec_main_boundary}

The proof relies on the so-called {\it prescribed level sets method}, which consists in using test functions for the Rayleigh quotient that are constant on the level lines of the Green function of $\Omega$ with pole $z_0$: $\psi_{z_0}$. The problem is then reduced to the study of the monotonicity of the first eigenvalue of a one-dimensional eigenvalue problem. The strategy is similar to that employed in \cite{CLPS_reverse}. We remark that from the scaling properties of the eigenvalues, it is sufficient to prove Theorem \ref{main_boundary} for $\Omega\in\mathcal M_1$, that is, when the gaussian curvature is bounded above by $1$. Also, by gauge invariance, it is not restrictive to assume $\nu\in\left(0,\frac{1}{2}\right]$. Through all the proof we will also set 
$$
M\doteq |\Omega|=|D^{\star}|.
$$
\subsection{Reduction to a one dimensional problem}\label{subsect:optimpole}
  Let $\nu\in(0,\frac{1}{2}]$ and let $z_0\in\Omega$ be fixed. Let $\psi_{z_0}$ be the Green function in $\Omega$ with pole $z_0$ defined in \eqref{green_boundary}. Recall that
$$
A_{z_0}^{(\nu)}=-2\pi\nu\star d\psi_{z_0}.
$$ 
Let $D^{\star}$ be a spherical cap in $\mathbb S^2$ with $|D^{\star}|=|\Omega|=M<4\pi$, and let $x_0$ be its center. By $\psi_{x_0}$ we denote the Green function of $D^{\star}$ with pole $x_0$ and by $A_{x_0}^{(\nu)}$ the potential in $D^{\star}$ defined by $A_{x_0}^{(\nu)}=-2\pi\nu\star d\psi_{x_0}$. We have $A^{(\nu)}_{x_0}=\nu d\theta$. Since  $\nu$, $\Omega, D^{\star}$, $z_0$ and $x_0$ are fixed, to simplify the notation we shall write, through all the proof
$$
A\doteq A_{z_0}^{(\nu)}\,,\ \ \ A^{\star}\doteq A_{x_0}^{(\nu)}\,,\ \ \ \psi\doteq \psi_{z_0}\,,\ \ \ \psi^{\star}\doteq \psi_{x_0}.
$$
It is well-known that $\psi,\psi^{\star}$ have no critical points. Consider now the function space
$$
\mathcal R(\Omega)\doteq\{u:u=g\circ\psi\,,\ g\in H^1(0,\infty)\}\subset H^1_{A}.
$$
The space $\mathcal R(\Omega)$ consists of all functions in $H^1_{A}(\Omega)$ that are constant on the level lines of $\psi$. We set
\begin{equation}
\kappa_1(\Omega,A)\doteq\min_{0\ne u\in\mathcal R(\Omega)}\frac{
\int_{\Omega}|d^Au|^2}{\int_{\Omega}|u|^2}.
\end{equation}
From the min-max principle we have
\begin{equation}\label{ineq1}
\mu_1(\Omega,A)\leq\kappa_1(\Omega,A).
\end{equation}
Now, if $D^{\star}$ is a disk in $\mathbb S^2$ of radius $R$ and the pole of the magnetic potential coincides with its center, then the first eigenfunction is real and radial (see e.g., \cite[Appendix B.1]{CPS22}), hence
\begin{equation}\label{eq1}
\mu_1(D^{\star},A^{\star})=\kappa_1(D^{\star},A^{\star}).
\end{equation}
In order to complete the proof of Theorem \ref{main_boundary}, it is enough to prove
\begin{equation}
\kappa_1(\Omega,A)\leq\kappa_1(D^{\star},A^{\star})
\end{equation}
and show that the equality holds only if $\Omega=D^\star$ and $A=A^\star$.

\medskip

To do so, we characterize $\kappa_1(\Omega,A)$ as the minimizer of a $1$-dimensional Rayleigh quotient associated with a Sturm-Liouville problem. Define a function $G_{z_0}:(0,M)\to(0,\infty)$ by
\begin{equation}\label{Gz0}
G_{z_0}(a)\doteq\int_{\psi_{z_0}=\beta_{z_0}(a)}\frac{1}{|d\psi_{z_0}|},
\end{equation}
where $\beta_{z_0}(a)$ is defined by $|\{\psi_{z_0}>\beta_{z_0}(a)\}|=a$. Analogously we define the function $G_{x_0}(a)$ (and $\beta_{x_0}(a)$) using the Green function $\psi_{x_0}$ of $D^{\star}$ with $x_0$ the center of $D^{\star}$.

\medskip

To keep the same notation as above, we will just write, for the rest of the proof
$$
G\doteq G_{z_0}\,,\ \ \ G^{\star}\doteq G_{x_0}.
$$
We recall that $G(a)\sim 4\pi a$ as $a\to 0$. This follows from the fact that $\psi\sim-\frac{1}{2\pi}\log r$ as $r\to 0$, where $r$ is the geodesic distance from $z_0$, and that $\psi+\frac{1}{2\pi}\log r$ is a smooth function near $z_0$. The same holds true for $G^{\star}$. However, for $G^{\star}$ we have an explicit expression:
$$
G^{\star}(a)=a(4\pi-a),
$$
which follows from the {\it geometric isoperimetric inequality} \eqref{geo_iso}, valid for spherical domains (see \eqref{ineq_G}).
\begin{lem}\label{1dreduction}
We have that
\begin{equation}\label{K1}
\kappa_1(\Omega,A)=\min_{0\ne f\in\mathcal F}\frac{\int_0^M\left(G(a)f'(a)^2+\frac{4\pi^2\nu^2}{G(a)}f(a)^2\right)da}{\int_0^Mf(a)^2da}
\end{equation}
where $\mathcal F=\{f\in L^2(0,M):\sqrt{G}f',f/\sqrt{G}\in L^2(0,M)\}$. In particular, $\kappa_1(\Omega,A)$ is the first eigenvalue of the following Sturm-Liouville problem in $(0,M)$:
\begin{equation}\label{SL1}
\begin{cases}
-(Gf')'+\frac{4\pi^2\nu^2}{G}f=\kappa f\,, & {\rm in\ }(0,M)\,,\\
\lim_{a\to 0}G(a)f'(a)=f'(M)=0.
\end{cases}
\end{equation}
The same statements hold for $\kappa_1(D^{\star},A^{\star})$, replacing $G$ with $G^{\star}$.
\end{lem}
\begin{proof}
Let $u$ be real-valued and of the form $u=g\circ\psi$. Then, since $|A|^2=4\pi^2\nu^2|d\psi|^2$:
$$
|d^Au|^2=(g'\circ\psi)^2|d\psi|^2+(g\circ\psi)^2|A|^2=\left((g'\circ\psi)^2+4\pi^2\nu^2(g\circ\psi)^2\right)|d\psi|^2
$$
hence by the coarea formula we get
$$
\int_{\Omega}|d^Au|^2=\int_0^{\infty}\left(g'(t)^2+4\pi^2\nu^2g(t)^2\right)\int_{\psi=t}|d\psi|\,dt=\int_0^{\infty}\left(g'(t)^2+4\pi^2\nu^2g(t)^2\right)dt.
$$
The last identity follows since
$$
\int_{\psi=t}|d\psi|=\int_{\psi=t}\partial_N\psi=\int_{\{\psi>t\}}\Delta\psi=1,
$$
where $N$ is the inward unit normal to the set $\{\psi>t\}$. Now, let
$$
\alpha(t)\doteq |\{\psi>t\}|=\int_t^{\infty}\int_{\psi=s}\frac{1}{|d\psi|}\,dt
$$
so that
$$
\alpha'(t)=-\int_{\psi=t}\frac{1}{|d\psi|}.
$$
Recall that $\psi$ has no critical points in $\overline\Omega\setminus\{z_0\}$, hence $\alpha:(0,\infty)\to (0,M)$ is smooth, strictly decreasing, and admits a smooth inverse $\beta:(0,M)\to (0,\infty)$. We change variable and set
$$
t=\beta(a)\,,\ \ \ f=g\circ\beta,
$$
hence $g(t)=f(a)$. We have that
$$
\beta'(a)=-\frac{1}{G(a)}
$$
where we recall that $G(a)=\int_{\psi=\beta(a)}\frac{1}{|d\psi|}$. We conclude that
$$
\int_{\Omega}|d^Au|^2=\int_0^{\infty}\left(g'(t)^2+4\pi^2\nu^2 g(t)^2\right)dt=\int_0^M\left(G(a)f'(a)^2+\frac{4\pi^2\nu^2}{G(a)}f(a)^2\right)da.
$$
Analogously, we have that
$$
\int_{\Omega}u^2=\int_0^{\infty}g(t)^2\int_{\psi=t}\frac{1}{|d\psi|^2}\,dt=\int_0^Mf(a)^2da.
$$
This proves \eqref{K1}. Now, from \eqref{K1} it is standard to check that $\kappa_1(\Omega,A)$ is the first eigenvalue of the Sturm-Liouville problem \eqref{SL1}.
\end{proof}

\subsection{Monotonicity of the function $G$}

We prove now that $G\geq G^{\star}$ and discuss the equality case.

\begin{lem}\label{lem:isoper} One  has on $(0,M)$:
    \begin{equation*}
     G\geq G^{\star} 
    \end{equation*}
    with equality if and only if $\Omega=D^{\star}$ and $A=A^{\star}$.
\end{lem}
\begin{proof}
For any $t\in(0,\infty)$ we have, by Cauchy-Schwarz inequality 
\begin{multline}\label{chain}
|\{\psi=t\}|=\int_{\psi=t}|d\psi|^{1/2}|d\psi|^{-1/2}\leq\left(\int_{\psi=t}|d\psi|\right)^{1/2}\left(\int_{\psi=t}\frac{1}{|d\psi|}\right)^{1/2}\\=\left(\int_{\psi=t}\frac{1}{|d\psi|}\right)^{1/2}
\end{multline}
with equality if and only if $|d\psi|$ is constant on $\{\psi=t\}$. Changing variables, and passing to the variable $a=\alpha(t)=|\{\psi>t\}|$, we have that \eqref{chain} reads
$$
|\{\psi=\beta(a)\}|^2\leq G(a),
$$
where $\beta$ is the inverse of $\alpha$. Now, if $\Omega$ is a simply connected Riemannian surface with boundary of area $\mathcal A$, boundary length $\ell$, and gaussian curvature bounded above by $1$, we have the following isoperimetric inequality:
\begin{equation}\label{geo_iso}
\ell^2\geq \mathcal A(4\pi-\mathcal A),
\end{equation}
with equality if and only if $\Omega$ is a disk of constant curvature $1$. We refer e.g., to \cite{chavel_isop} for a proof. Therefore we conclude that
\begin{equation}\label{ineq_G}
G(a)\geq a(4\pi-a)=G^{\star}(a),
\end{equation}
with equality if and only if each level set $\psi=t$ is a circle and $|d\psi|$ is constant on $\psi=t$. This last condition imposes that $\psi$ must be a radial function, hence $\Omega=D^{\star}$ and $A=A^{\star}$.
 \end{proof}
\subsection{Optimization of a 1-d problem}
Next we consider the following number
$$
\kappa_1(G)\doteq\min_{0\ne f\in\mathcal F}\frac{\int_0^M\left(G(a)f'(a)^2+\frac{4\pi^2\nu^2}{G(a)}f(a)^2\right)da}{\int_0^Mf^2(a)da},
$$
where $\mathcal F=\{f\in L^2(0,M):\sqrt{G}f',f/\sqrt{G}\in L^2(0,M)\}$ and $G:(0,M)\to (0,\infty)$ is any smooth, positive function satisfying
\begin{equation}\label{asymptotic}
G(a)\sim 4\pi a\,,\ \ \ a\to 0.
\end{equation}
Note that $\kappa_1(\Omega,A^{(\nu)}_{z_0})=\kappa_1(G_{z_0})$ where $G_{z_0}$ has been defined in \eqref{Gz0}. Recall also that $G^{\star}=a(4\pi-a)$. We want to optimize $\kappa_1(G)$ with respect to the weight $G$. In particular we have the following lemma. 
\begin{lem}\label{lem:optimsturmmliouville}
    Let $0< G_1\leq G_2$ on $(0,M)$, both satisfying \eqref{asymptotic}. Then
    \begin{equation*}
      \kappa_1(G_2)\leq \kappa_1(G_1).  
    \end{equation*}
    If $G_1<G_2$ on a set of positive measure, the inequality is strict.
\end{lem}
\begin{proof}
For $t\in[0,1]$ consider
$$
G_t:=(1-t)G_1+tG_2.
$$
We  compute the derivative  of $\kappa_1(G_t)$ with respect to $t$.  Since the eigenvalue is simple, we can apply Feynman-Hellmann formula \cite[VII-\S4, p.408, n. 4.56]{kato1} to get
$$
\frac{d}{dt}\kappa_1(G_t)=\int_0^M\left(G_t(a)^2f_t'(a)^2-4\pi^2\nu^2 f_t(a)^2\right)\frac{G_2(a)-G_1(a)}{G_t(a)^2}\,da
$$
where $f_t$ is the first positive $L^2$-normalized eigenfunction corresponding to the eigenvalue $\kappa_1(G_t)$. In order to conclude we have to prove that
$$
G_t^2f_t'^2-4\pi^2\nu^2f_t^2< 0.
$$
Set
$$
R=\frac{G_tf_t'}{f_t}.
$$
We see that $R$ is smooth in $(0,M)$, and moreover, from the Neumann condition at $M$, we have that $R(M)=0$. We have to show that
$$
|R(a)|< 2\pi\nu\ \ \ \forall a\in(0,M).
$$
{\bf First step.} Assume by contradiction that there exists $a_0\in (0,M)$ such that $R(a_0)\geq 2\pi\nu$. Differentiating $R$ we get
\begin{equation}\label{eq:DerivR}
   R'=\frac{4\pi^2\nu^2-R^2}{G_t}-\kappa_1(G_t). 
\end{equation}
In particular $R'(a_0)<0$, and this implies, along with \eqref{eq:DerivR}, that $R$  is decreasing (and positive) in the whole interval $(0,a_0)$. Hence limit $\lim_{a\rightarrow 0^+} R(a)$ exists, so we have two cases

\smallskip

\textbf{Case $1$}: $\lim_{a\rightarrow 0^+} R(a)$ is finite. In particular then, $\lim_{a\rightarrow 0^+} R(a)>2\pi\nu$. We integrate \eqref{eq:DerivR} and we obtain 
\begin{equation*}
R(a_0)-R(a)=\int_a^{a_0}R'(s)ds=\int_a^{a_0}\frac{4\pi^2\nu^2-R^2}{G_t}-\kappa_1(G_t)ds.   
\end{equation*}
But now we recall that $G_t(a)\sim 4\pi a$ when $a\rightarrow 0^+$, hence, since $\lim_{a\rightarrow 0^+} R(a)>2\pi\nu$, we get that the right-hand side of the above equation is unbounded (tends to $-\infty$) as $a\to 0^+$, while the left-hand side remains bounded.

\smallskip

\textbf{Case $2$}: $\lim_{a\rightarrow 0^+} R(a)=+\infty$. From \eqref{eq:DerivR}, recalling that $G_t(a)\sim 4\pi a$ as $a\to 0^+$, we obtain 
\begin{equation*}
  \lim_{a\rightarrow 0^+} \frac{a R'}{R^2}=-\frac{1}{4\pi}.  
\end{equation*}
Therefore, there exists $\bar a>0$ and $\epsilon>0$ such that 
\begin{equation*}
  \frac{1}{a}\left (-\frac{1}{4\pi}-\epsilon \right )\leq \left (-\frac{1}{R}\right )'\leq \frac{1}{a}\left (-\frac{1}{4\pi}+\epsilon \right) 
\end{equation*}
on $(0,\bar a)$. We can integrate both sides in $(a,\bar a)$, and we obtain 
\begin{equation*}
 \left (-\frac{1}{4\pi}-\epsilon \right ) \ln({\frac{\bar{a}}{a}})\leq \frac{1}{R(a)}-\frac{1}{R(\bar a)}\leq \left (-\frac{1}{4\pi}+\epsilon \right ) \ln({\frac{\bar{a}}{a}}) 
\end{equation*}
In particular from the above inequality we can conclude that $\lim_{a\rightarrow 0^+}R(a)=0$ which is a contradiction.

We conclude that there is no $a_0\in (0,M)$ such that $R(a_0)\geq 2\pi\nu$.

\medskip

\noindent{\bf Second step.} It remains to consider the case in which there exists $a_0\in(0,M)$ such that $R(a_0)\leq-2\pi\nu$. This implies, by \eqref{eq:DerivR}, that $R'<0$ on $(a_0,M)$, and hence $R(M)<-2\pi\nu$, which is a contradiction with $R(M)=0$.
\end{proof}



\subsection{Conclusion of the proof}
The proof of Theorem \ref{main_boundary} now follows from \eqref{ineq1}, \eqref{eq1} and Lemmas \ref{lem:isoper} and \ref{lem:optimsturmmliouville}:
\begin{align*}
\mu_1(\Omega,A)&\leq\kappa_1(\Omega,A)\ \ \ {\rm by\ \eqref{ineq1}}\\
&\leq \kappa_1(D^{\star},A^{\star})\ \ \ {\rm by\ Lemmas\  \ref{lem:isoper}\ and\ \ref{lem:optimsturmmliouville}}\\
&=\mu_1(D^{\star},A^{\star})\ \ \ {\rm by\ \eqref{eq1}}.
\end{align*}
If equality holds everywhere, then $\kappa_1(\Omega,A)=\kappa_1(D^{\star},A^{\star})$ which implies from Lemmas \ref{lem:isoper} and \ref{lem:optimsturmmliouville} that $\Omega=D^{\star}$ and $A=A^{\star}$.

\section{Proof of Theorems \ref{main_closed} and \ref{main_closed_2}}\label{sec_main_closed}
The proof is similar to that of Theorem \ref{main_boundary}. However, note that Theorems \ref{main_closed} and \ref{main_closed_2} give an upper bound which is not isoperimetric. In particular, in the case of a closed surface $\Sigma$, the bound on the  gaussian curvature implies that $|\Sigma|\geq |\mathbb S^2_K|$, and if $\Sigma\ne\mathbb S^2_K$, the inequality is strict. As in Section \ref{sec_main_boundary}, we may suppose that $K=1$. We will prove Theorem \ref{main_closed} first.

\subsection{Proof of Theorem \ref{main_closed}: reduction to a 1-d problem}
 By gauge invariance, we may always assume $\nu\in(0,\frac{1}{2}]$.   Let $\psi_{z_1,z_2}$ be the Green function of $\Sigma$ with poles $z_1,z_2$ (defined in \eqref{green_closed}) and $A^{(\nu)}_{z_1,z_2}=-2\pi\nu\star d\psi_{z_1,z_2}$. We also set
$$
M\doteq |\Sigma|
$$
and define
\begin{equation}\label{Gz1z2}
G_{z_1,z_2}(a)\doteq\int_{\psi_{z_1,z_2}=\beta_{z_1,z_2}(a)}\frac{1}{|d\psi_{z_1,z_2}|}
\end{equation}
where $\beta_{z_1,z_2}(a)$ is defined by $|\{\psi_{z_1,z_2}>\beta_{z_1,z_2}(a)\}|=a$. The behavior of $G_{z_1,z_2}$ near the poles is as follows: $G_{z_1,z_2}(a)\sim 4\pi a$ as $a\to 0^+$ and $G_{z_1,z_2}(a)\sim 4\pi(M-a)$ as $a\to M^-$. As in the previous section, to simplify the notation we let
 $$
A\doteq A_{z_1,z_2}^{(\nu)}\,,\ \ \ \psi\doteq \psi_{z_1,z_2}\,,\ \ \ G\doteq G_{z_1,z_2}.
$$
We also set
 $$
A^{\star}\doteq A_{x_0}^{(\nu)}\,,\ \ \ \psi^{\star}\doteq \psi_{x_0}\,,\ \ \ G^{\star}\doteq G_{x_0}=a(4\pi-a),
$$
where $\psi_{x_0}$ is the Green function on $\mathbb S^2$ with two antipodal poles $x_0,-x_0$, $A^{(\nu)}_{x_0,}=-2\pi\nu\star d\psi_{x_0}=\nu d\theta$ and $G_{x_0}$ is defined by \eqref{Gz1z2} with $\psi_{x_0}$.
 Consider now the function space
$$
\mathcal R(\Sigma)\doteq\{u:u=g\circ\psi\,,\ g\in H^1(\mathbb R)\}\subset H^1_{A}(\Sigma)
$$
and the number
\begin{equation}\label{w1}
w_1(\Omega,A)\doteq\min_{0\ne u\in\mathcal R(\Sigma)}\frac{\int_{\Sigma}|d^Au|^2}{\int_{\Sigma}|u|^2}.
\end{equation}
From the min-max principle we have
\begin{equation}\label{ineqS1}
\mu_1(\Sigma,A)\leq w_1(\Sigma,A).
\end{equation}
We have the following lemma, whose proof is analogous to that of Lemma \ref{1dreduction}
\begin{lem}\label{lem_w1}
We have that
\begin{equation}\label{w1_variable_a}
w_1(\Sigma,A)=\min_{0\ne f\in\mathcal F}\frac{\int_0^M\left(G(a)f'(a)^2+\frac{4\pi^2\nu^2}{G(a)}f(a)^2\right)da}{\int_0^Mf(a)^2da},
\end{equation}
where $\mathcal F=\{f\in L^2(0,M):\sqrt{G}f',f/\sqrt{G}\in L^2(0,M)\}$. In particular $w_1(\Sigma,A)$ is the first eigenvalue of the following Sturm-Liouville problem
\begin{equation}\label{SL2}
\begin{cases}
-(Gf')'+\frac{4\pi^2\nu^2}{G}f=w f\,, & {\rm in\ }(0,M)\,,\\
\lim_{a\to 0^+}G(a)f'(a)=\lim_{a\to M^-}G(a)f'(a)=0.
\end{cases}
\end{equation}
\end{lem}
\begin{rem}
Note that in the definition of $\mathcal R(\Sigma)$ we are considering functions $g\in H^1(\mathbb R)$ since now the Green function on $\Sigma$ with poles $z_1,z_2$ assumes all real values (it goes to $+\infty$ and $-\infty$ at $z_1$, $z_2$, respectively).
\end{rem}
\subsection{Optimization of the $1$-d problem}
Let now $G$ be any smooth, positive function in $(0,M)$ satisfying 
\begin{equation}\label{asymptotic_closed}
G(a)\sim 4\pi a \ {\rm as}\ a\to 0^+\ \ \ {\rm and\ \ \ }G(a)\sim 4\pi(M-a) \ {\rm as}\ a\to M^-,
\end{equation} and consider
$$
w_1(G)\doteq\min_{0\ne f\in\mathcal F}\frac{\int_0^M\left(G(a)f'(a)^2+\frac{4\pi^2\nu^2}{G(a)}f(a)^2\right)da}{\int_0^Mf^2(a)da},
$$
where $\mathcal F=\{f\in L^2(0,M):\sqrt{G}f',f/\sqrt{G}\in L^2(0,M)\}$. Note that $w_1(\Sigma,A^{(\nu)}_{z_1,z_2})=\kappa_1(G_{z_1,z_2})$. We have the following
\begin{lem}\label{lem:optimsturmmliouville_2}
    Let $0< G_1\leq G_2$ on $(0,M)$, both satisfying \eqref{asymptotic_closed}. Then
    \begin{equation*}
      w_1(G_2)\leq w_1(G_1).  
    \end{equation*}
    If $G_1<G_2$ on a set of positive measure, the inequality is strict.
\end{lem}
\begin{proof}
The proof is identical to that of Lemma \ref{lem:optimsturmmliouville}.
The only difference from the proof of Lemma \ref{lem:optimsturmmliouville} lies in the second step, namely the treatment of the case $R\leq-2\pi\nu$ since now also $a=M$ is a singular point. 

\smallskip

\noindent{\bf Second step.} Assume by contradiction that there exists a point $a_0\in(0,M)$ such that $R(a_0)\leq -2\pi \nu$. We see from \eqref{eq:DerivR} that $R'(a_0)<0$, and this implies, still using \eqref{eq:DerivR}, that $R$  is decreasing in the whole interval $(a_0,M)$. Hence limit $\lim_{a\rightarrow M^-} R(a)$ exists, so we have two cases
\smallskip

\textbf{Case $1$}: $\lim_{a\rightarrow M^-} R(a)$ is finite. In particular then, $\lim_{a\rightarrow M^-} R(a)<-2\pi\nu$. We integrate \eqref{eq:DerivR} and we obtain 
\begin{equation*}
R(a)-R(a_0)=\int_{a_0}^{a}R'(s)ds=\int_{a_0}^{a}\frac{4\pi^2\nu^2-R^2}{G_t}-\kappa_1(G_t)ds.   
\end{equation*}
But now we recall that $G_t(a)\sim 4\pi(M-a)$ when $a\rightarrow M^-$, hence, since $\lim_{a\rightarrow M^-} R(a)<-2\pi\nu$, we get that the right-hand side of the above equation is unbounded as $a\to M^-$, while the left-hand side remains bounded.

\smallskip

\textbf{Case $2$}: $\lim_{a\rightarrow M^-} R(a)=-\infty$. From \eqref{eq:DerivR}, recalling that $G_t(a)\sim 4\pi (M-a)$ as $a\to M^-$, we obtain 
\begin{equation*}
  \lim_{a\rightarrow M^-} \frac{(M-a) R'}{R^2}=-\frac{1}{4\pi}.  
\end{equation*}
Therefore, there exists $\bar a>0$ and $\epsilon>0$ such that 
\begin{equation*}
  \frac{1}{(M-a)}\left (-\frac{1}{4\pi}-\epsilon \right )\leq \left (-\frac{1}{R}\right )'\leq \frac{1}{(M-a)}\left (-\frac{1}{4\pi}+\epsilon \right) 
\end{equation*}
on $(\bar a,4\pi)$. We can integrate both sides in $(\bar a,a)$, and we obtain 
\begin{equation*}
 \left (-\frac{1}{4\pi}-\epsilon \right ) \ln\left({\frac{M-\bar{a}}{M-a}}\right)\leq \frac{1}{R(a)}-\frac{1}{R(\bar a)}\leq \left (-\frac{1}{4\pi}+\epsilon \right ) \ln\left({\frac{M-\bar{a}}{M-a}}\right)
\end{equation*}
In particular from the above inequality we can conclude that $\lim_{a\rightarrow M^-}R(a)=0$ that is a contradiction. The conclusion of Lemma \ref{lem:optimsturmmliouville_2}  is now the same as in Lemma \ref{lem:optimsturmmliouville}.
\end{proof}

\subsection{When $\abs{\Sigma}>\abs{\mathbb S^2}$: a sequence of weight functions}
Due to the volume bound $|\Sigma|\geq |\mathbb S^2|$, we cannot directly compare $w_1(G)$ with $w_1(G^{\star})$ when $\abs{\Sigma}>\abs{\mathbb S^2}$ simply because $G,G^{\star}$ are defined in different intervals: $(0,M)$ and $(0,4\pi)$, respectively, with $4\pi< M$ (recall, $M=|\Sigma|$). Moreover $M=4\pi$ if and only if $\Sigma=|\mathbb S^2|$. 

We define a sequence of weights $G^{\star}_\eps$ as follows. Let $\eps_0>0$ be a sufficiently small number to be fixed later, and let $\eps\in(0,\eps_0)$. Let
$$
a_\eps\doteq 2\pi+\sqrt{4\pi^2-\eps}\in(2\pi,4\pi)
$$
and
$$
b_{\eps}\doteq 2\pi-\sqrt{4\pi^2-\eps}\in(0,2\pi).
$$
These are the solutions of the equation $G^{\star}(a)=\eps$. Now, from Lemma \ref{lem:isoper} we have that
$$
G(a)\geq G^{\star}(a){\rm\  in\ }(0,4\pi){\rm \ \ \ and\ \ \ }G(a)\geq G^{\star}(M-a){\rm\  in\ }(M-4\pi,M)
$$
and moreover $G>0$ in $(0,M)$. Then there exists $\eps_0>0$ such that for all $\eps\in(0,\eps_0)$, we have $G\geq G^{\star}_\eps$, where
$$
G^{\star}_\eps=\begin{cases}
G^{\star}\,, & {\rm in\ } (0,a_\eps)\\
\eps & {\rm in\ } (a_\eps,M-b_\eps)\\
G^{\star}(a-M)\,, & {\rm in\ }(M-b_\eps,M).
\end{cases}
$$
See Figure \ref{fig1}. The function $G^{\star}_\eps$ is continuous on $(0,M)$ and converges uniformly to $G^{\star}$ in $(0,4\pi)$ and to $0$ in $(4\pi,M)$.

\begin{figure}[ht] 
\begin{center} 
  \includegraphics[width=12cm]{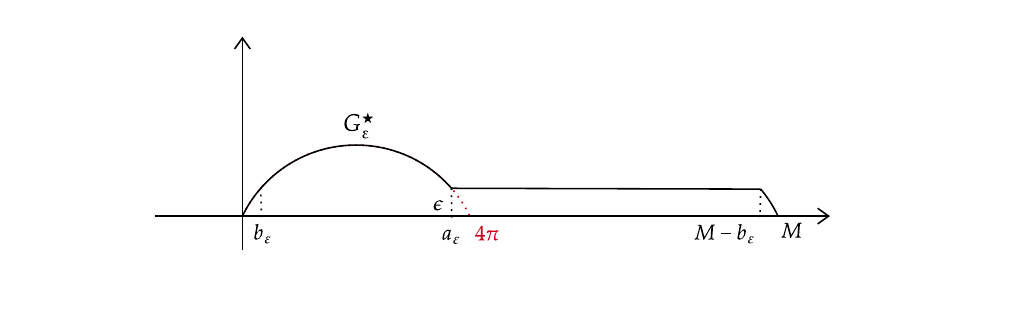}\\ 
  \caption{the function $G^{\star}_\eps$  }
  \label{fig1}
\end{center} 
\end{figure}

\subsection{Proof of the upper bound in Theorem \ref{main_closed}}
From Lemma \ref{lem:optimsturmmliouville_2} we deduce that, since $G\geq G^{\star}_\eps$ for all $\eps\in(0,\eps_0)$ we see:
\begin{equation}\label{ineq1closed}
w_1(G)\leq w_1(G^{\star}_\eps).
\end{equation}
Now, the upper bound \eqref{ineq_main_closed} in Theorem \ref{main_closed} is  a consequence of the following lemma and the fact that $\mu_1(\mathbb S^2,A^{\star})=\nu(\nu+1)$ if $\nu\in(0,\frac 12]$. The explicit expression of $\mu_1(\mathbb S^2,A^\star)$ is proved in Appendix \ref{appendix}. Recall that by $\lambda_1(\Omega,A)$ we denote the first magnetic Dirichlet eigenvalue on $\Omega$ with potential $A$, see \eqref{ev_dir}.

\begin{lem}
The following facts hold
\begin{enumerate}[i)]
\item $w_1(G^{\star}_\eps)\leq\lambda_1(D_{a_\eps},A^{\star})$,
where $\lambda_1(D_{a_\eps},A^{\star})$ is the first magnetic Dirichlet eigenvalue of a geodesic disk $D_{a_\eps}\subset\mathbb S^2$ of area $a_\eps$ with center $x_0$ and with $A^{\star}=-2\pi\nu\star d\psi_{x_0}=\nu d\theta$
\item $\lim_{\eps\to 0}\lambda_1(D_{a_{\eps}},A^{\star})=\mu_1(\mathbb S^2,A^{\star})$.
\end{enumerate}
\end{lem}
\begin{proof}
From the min-max principle we have that
\begin{multline}
w_1(G^{\star}_\eps)=\min_{0\ne f\in\mathcal F}\frac{\int_0^M\left(G^{\star}_\eps(a)f'(a)^2+\frac{4\pi^2\nu^2}{G^{\star}_\eps(a)}f(a)^2\right)da}{\int_0^Mf^2(a)da}\\
\leq\min_{\substack{0\ne f\in\mathcal F\\f\equiv 0{\rm\ on\ }(a_\eps,M)}}\frac{\int_0^M\left(G^{\star}_\eps(a)f'(a)^2+\frac{4\pi^2\nu^2}{G^{\star}_\eps(a)}f(a)^2\right)da}{\int_0^Mf^2(a)da}\\
=\min_{\substack{0\ne f\in\mathcal F\\f(a_\eps)=0}}\frac{\int_0^{a_\eps}\left(G^{\star}(a)f'(a)^2+\frac{4\pi^2\nu^2}{G^{\star}(a)}f(a)^2\right)da}{\int_0^{a_\eps}f^2(a)da}=\lambda_1(D_{a_\eps},A^\star).
\end{multline}
The last equality holds since for the disk with Aharonov-Bohm potential with pole in the center and flux $\nu\in(0,\frac 12]$, there exists a first real and radial eigenfunction, as in the Neumann case. This proves i). Point ii) is proved in Lemma \ref{lem_app_conv}. Namely, in Lemma \ref{lem_app_conv} we show  that the first magnetic Dirichlet eigenvalue on a spherical cap with pole at its center coverges to the first eigenvalue on the sphere with two opposite poles when the cap invades the whole sphere.
\end{proof}

\subsection{The equality case in Theorem \ref{main_closed}}
Suppose now that $\mu_1(\Sigma,A)=\mu_1(\mathbb S^2,A^{\star})$. Let $\eps_0>\eps_1>\eps_2\cdots>\eps_n\to 0$ as $n\to\infty$.

Then
\begin{multline}
\mu_1(\mathbb S^2,A^{\star})=\mu_1(\Sigma,A)\leq w_1(G)\leq w_1(G^{\star}_{\eps_1})<w_1(G^{\star}_{\eps_2})\\<\cdots<w_1(G^{\star}_{\eps_n})\leq\lambda_1(D_{a_{\eps_n}},A^\star)\to\mu_1(\mathbb S^2,A^{\star}).
\end{multline}
By Lemma \ref{lem:optimsturmmliouville_2} we have that the inequalities $w_1(G^{\star}_{\eps_n})<w_1(G^{\star}_{\eps_{n+1}})$ are strict, since by construction $G^{\star}_{\eps_n}>G^{\star}_{\eps_{n+1}}$ on a set of positive measure. Hence all the inequalities must be equalities, which is not possible, unless $M=4\pi$ and hence $\Sigma=\mathbb S^2$. Thus we are in the case of the sphere $\mathbb S^2$ with two poles $x_0,y$ as in Corollary \ref{cor_2poles}. If equality holds, then $G=G^{\star}$ in $(0,4\pi)$ and this implies that $y=-x_0$ by Lemma \ref{lem:isoper}.

\subsection{Proof of Theorem \ref{main_closed_2}}
The proof of Theorem \ref{main_closed_2} is exactly the same as that of Theorem \ref{main_closed}, except that now $G^{\star}_{\eps}$ is defined as
$$
G^{\star}_\eps=\begin{cases}
G^{\star}\,, & {\rm in\ } (0,a_\eps)\\
\eps\,, & {\rm in\ } (a_\eps,M).
\end{cases}
$$
for all positive $\eps\leq\eps_0$, for a suitable $\eps_0>0$. The equality case is discussed as in Theorem \ref{main_closed}.

\subsection{Hersch's Theorem fails: sharpness of Theorem \ref{main_closed}}\label{sub:sharp}
In this section we want to show that without an assumption on the Gaussian curvature, we cannot have a result like Hersch's Theorem for the usual Laplace-Beltrami operator.

\begin{thm}\label{nohersch}
One has
$$
\sup_{\Sigma,\, z_1,z_2\in\Sigma}|\Sigma|\,\mu_1(\Sigma,A^{(\nu)}_{z_1,z_2})=+\infty
$$
for all $\nu\notin\mathbb Z$.
\end{thm}
\begin{proof}
We prove the theorem with the following family of surfaces. For any $L>0$  let $\Sigma_{L}$ be the cigar-like surface obtained by gluing two hemispheres $\Omega^+$, $\Omega^-$ of radius $1$ to the boundary components of the cylinder $\Gamma_{L}=\{(x,y,z)\in\mathbb R^3:x^2+y^2=1,0\leq z\leq L\}$, see Figure \ref{fig_cygar}. The surface $\Sigma_{L}$ is punctured at the poles $p^+,p^-$ of the hemispheres. Let $A^{(\nu)}_{p^+,p^-}$ be the Aharonov-Bohm potential with poles $p^+,p^-$ and flux $\nu$.
By gauge invariance, we can take $\nu\in(0,\frac 12]$.
We show that
\begin{equation}\label{lower}
\mu_1(\Sigma_{L},A^{(\nu)}_{p^+,p^-})\geq\nu^2,
\end{equation} 
independently on $L$. In fact, from Appendix \ref{appendix} we have that $\mu_1(\Omega^{\pm},A^{(\nu)}_{p^+,p^-})=\nu(\nu+1)$, while from Appendix \ref{appendixb} we have that $\mu_1(\Gamma_{L},A^{(\nu)}_{p^+,p^-})=\nu^2$. To simplify the notation, for the rest of the proof we shall write $A=A^{(\nu)}_{p^+,p^-}$. Then for all $u\in H^1_A(\Sigma_{L})$ we have
$$
\int_{\Gamma_{L}}|d^Au|^2\geq\nu^2\int_{\Gamma}|u|^2
$$
and
$$
\int_{\Omega^{\pm}}|d^Au|^2\geq\nu(\nu+1)\int_{\Omega^{\pm}}|u|^2.
$$
Summing up, since $\nu(\nu+1)\geq\nu^2$, we get 
$$
\int_{\Sigma_{L}}|d^Au|^2\geq\nu^2\int_{\Sigma_{L}}|u|^2
$$
for any $u\in H^1_A(\Sigma_{L})$, hence \eqref{lower}.
It follows that
$$
\abs{\Sigma_L}\mu_1(\Sigma_{L},A^{(\nu)}_{p^+,p^-})\geq 2\pi\nu^2(L+2)
$$
which grows linearly to infinity as $L\to\infty$.
The assertion follows.
\end{proof}

\begin{figure}[ht] 
\begin{center} 
  \includegraphics[width=\textwidth]{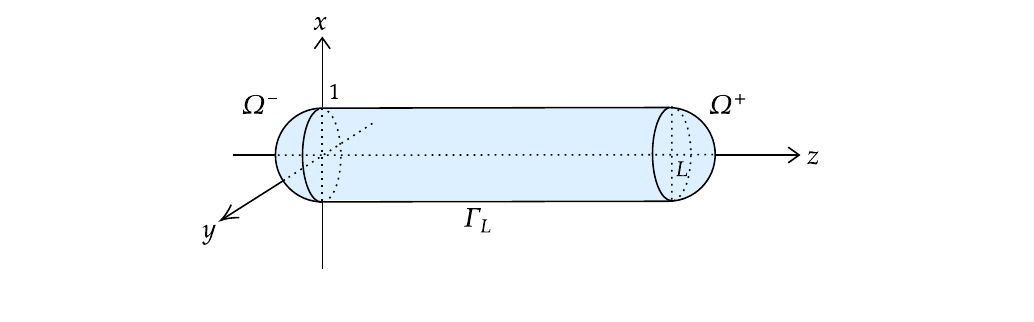}\\ 
  \caption{The cigar-like surface $\Sigma_{L}$}
  \label{fig_cygar}
\end{center} 
\end{figure}

A final remark is that in Theorem \ref{nohersch} we give examples of closed surfaces for which we have lower bounds of the type $\nu^2K$ which are independent on the area, where $K$ is an upper bound on the Gaussian curvature ($K=\frac{1}{R^2}$ for $\Sigma_{L,R}$ in the proof of Theorem \ref{nohersch}).

\subsection{Sharpness of Theorem \ref{main_closed_2}} It is enough to consider the example in the previous subsection and remove one of the two spherical caps $\Omega^+$ or $\Omega^-$.

\appendix

\section{Spectrum of the punctured sphere, and more}\label{appendix}

\subsection{Spectrum of the punctured sphere}In this appendix  we compute explicitly the spectrum of the sphere with Aharonov-Bohm potential with two opposite poles and flux $\nu\notin\mathbb Z$.

Let $x_0\in\mathbb S^2$ be a fixed point. Consider polar coordinates $(r,\theta)$ based at $x_0$, where $r$ is the geodesic distance from $x_0$ and $\theta$ is the polar angle.  Let $A^{(\nu)}_{x_0}=\nu d\theta$. It is closed and has flux $\nu$ around any loop separating $x_0$ and $-x_0$.  

\begin{prop}\label{prop:specsphere}
    Let $\nu\in (0,\frac{1}{2}]$. The spectrum of the magnetic Laplacian with potential $A^{(\nu)}_{x_0}$ on $\mathbb S^2$ is described as follows:
    \begin{enumerate}[a)]
        \item if $\nu\in (0,\frac{1}{2})$ then the eigenvalues are given by the union of the following sequences of numbers (counted with multiplicities)
        \begin{align*}
            \{(n-\nu)(n-\nu+1)\}_{n\geq 1}& \text{ with multiplicity }n \\
            \{(\abs n+\nu)(\abs n+\nu+1)\}_{n\leq 0}& \text{ with multiplicity }|n|+1 
        \end{align*}
        \item if $\nu=\frac{1}{2}$ then the eigenvalues are given by the following sequence of numbers (counted with multiplicity) 
        \begin{equation*}
            \{(n-\nu)(n-\nu+1)\}_{n\geq 1}\text{ with multiplicity }2n 
        \end{equation*}
   
\item In particular, if $\nu\in (0,\frac  12)$ one has 
$$
\mu_1(\mathbb S^2,A^{(\nu)}_{x_0})=\nu(\nu+1)
$$
with eigenspace spanned by $u(r,\theta)=\sin^{\nu}(r)$.  If $\nu=\frac 12$ the lowest eigenvalue is $\frac 34$, it is double, and the eigenspace is spanned by $\sin^{\frac 12}(r)$, $e^{i\theta}\sin^{\frac 12}(r)$.
\item If $\nu\in\mathbb R$, $\nu\notin\mathbb Z$, formulae a), b) and c) hold with $\nu$ replaced by $\min_{m\in\mathbb Z}|m-\nu|$.

\end{enumerate}
 \end{prop}

When $\nu\in (0,\frac  12)$ the first few eigenvalues $\mu_k\doteq \mu_k(\mathbb S^2,A^{(\nu)}_{x_0})$ are then given by
$$
\begin{aligned}
&\mu_1=\nu(\nu+1), \quad \mu_2=(1-\nu)(2-\nu)\\
&\mu_3=\mu_4=(1+\nu)(2+\nu), \quad \mu_5=\mu_6=(2-\nu)(3-\nu) \quad \text{etc.}
\end{aligned}
$$

\begin{rem}
    Note that at the limit as $\nu\to 0$ we recover the spectrum of the Laplace-Beltrami operator with multiplicities: the eigenvalues $(n-\nu)(n-\nu+1)$ go to $n(n+1)$ with multiplicity $n$ ($n\geq 1$) and the eigenvalues $(\nu-n)(\nu-n+1)$ go to $|n|(|n|+1)$ with multiplicity $|n|+1$ ($n\leq 0$). Altogether, the eigenvalues are given by $\ell(\ell+1)$ with multiplicity $2\ell+1$, $\ell=0,1,2,\cdots$. 
\end{rem}

\begin{proof}
Through the proof we shall simply write $A\doteq A^{(\nu)}_{x_0}=\nu d\theta$. Writing the magnetic Laplacian in polar coordinates $(r,\theta)$ based at $x_0$ we get that, for any smooth, complex-valued function $u$,
\begin{equation}\label{magnetic_polar}
\Delta_Au=-\partial^2_{rr}u-\cot(r)\partial_ru-\frac{1}{\sin^2(r)}\partial^2_{\theta\theta}u+\frac{\nu^2}{\sin^2(r)}u+\frac{2i\nu}{\sin^2(r)}\partial_{\theta}u.
\end{equation}
We have no boundary, and the presence of the two singularities is taken into account in the choice of the energy space for the weak formulation, which is $H^1_A(\mathbb S^2):=\{u\in L^2(\mathbb S^2):|d^Au|\in L^2(\mathbb S^2)\}$. This imposes certain conditions at the two poles.

As customary, looking for solutions to $\Delta_Au=\mu u$ of the form $u(r,\theta)=e^{in\theta}v(r)$, $n\in\mathbb Z$, we end up with a family of Sturm-Liouville  problems indexed by $n\in\mathbb Z$:
\begin{equation}\label{SL_sphere}
\begin{cases}
v''(r)+\cot(r)v'(r)+\left(\mu-\frac{(n-\nu)^2}{\sin^2(r)}\right)v(r)=0\,, & {\rm\ in\ }(0,\pi)\,,\\
\lim_{r\to 0^+}\sin(r)v'(r)=\lim_{r\to\pi^-}\sin(r)v'(r)=0.
\end{cases}
\end{equation}
The boundary conditions at the singular endpoints encode the information on the energy space. Now, as shown e.g., in \cite[Lemma 23]{CPS22}, each problem \eqref{SL_sphere} admits an increasing sequence of simple eigenvalues
$$
0<\mu_{n0}<\mu_{n1}<\cdots<\mu_{nk}<\cdots\nearrow +\infty
$$
with associated eigenfunctions $v_{nk}$ which have $k$ zeros in $(0,\pi)$. First we will compute explicitly $\mu_{nk}$. Then we show that these numbers exhaust all the spectrum.

{\bf Step 1: compute $\mu_{nk}$.}  Recall that $\nu\notin\mathbb Z$. First, for all $n\in\mathbb Z$, we identify $\mu_{n0}$. In fact, take $v(r)=\sin^{|n-\nu|}(r)$. It does not change sign in $(0,\pi)$ and solves \eqref{SL_sphere} with $\mu=|n-\nu|(|n-\nu|+1)$. Thus we deduce that
$$
\mu_{n0}=|n-\nu|(|n-\nu|+1).
$$
When $\nu\in(0,\frac 12)$, the first eigenvalue is attained at $n=0$, and equals
$$
\mu_{00}=\nu(\nu+1);
$$
it is simple, and the corresponding eigenspace is spanned by $\sin^{\nu}(r)$. When $\nu=\frac 12$ we have that the first eigenvalue is
$$
\mu_{00}=\mu_{10}=\frac 34;
$$
it is double, and the corresponding eigenspace is spanned by $\sin^{\frac 12}(r)$ and $e^{i\theta}\sin^{\frac 12}(r)$.

Now we need to compute all the other eigenvalues. We first observe that a solution of the differential equation in \eqref{SL_sphere} is given by
$$
P_{\alpha}^{-|n-\nu|}(\cos(r)),
$$
 where $\alpha\in\mathbb R$ is related to the eigenvalue $\mu$ and is defined by
$$
\mu=\alpha(\alpha+1).
$$
Here $L^{b}_{a}$ denotes the Legendre function of degree $a$ and order $b$, see \cite[\S 14]{nist}. Note  that $P_{\alpha}^{-\alpha}(\cos(r))=c_{\alpha}\sin^{\alpha}(r)$, where $c_{\alpha}$ is a normalization constant.

Note also that there are two linearly independent solutions of the differential equation in \eqref{SL_sphere} which are possibly singular at the endpoints $0,\pi$.  
From \cite[\S 14.8, 14.8.1]{nist} we have
\begin{equation}\label{nist1}
P_{\alpha}^{-|n-1|}(x)\sim\frac{1}{\Gamma(|n-\nu|+1)}\left(\frac{1-x}{2}\right)^{\frac{|n-\nu|}{2}}\,,\ \ \ x\to 1^-
\end{equation}
hence, with our choice of taking $P_{\alpha}^{-|n-\nu|}$ we have
$$
\lim_{r\to 0^+}P_{\alpha}^{-|n-\nu|}(\cos(r))=0.
$$
Hence, an eigenvalue $\mu$ is determined by imposing
$$
\lim_{r\to \pi^-}P_{\alpha}^{-|n-\nu|}(\cos(r))=0.
$$
From \cite[\S 14.9, 14.9.7]{nist} we have
\begin{multline}\label{P-}
P_{\alpha}^{-|n-\nu|}(-x)=\frac{\sin(\nu\pi)}{\sin(|n-\nu|\pi)}P^{-|n-\nu|}_{\alpha}(x)\\
-\frac{\sin((|n-\nu|-\alpha)\pi)\Gamma(\alpha-|n-\nu|+1)}{\sin(|n-\nu|\pi)\Gamma(\alpha+|n-\nu|+1)}P^{|n-\nu|}_{\alpha}(x).
\end{multline}
Note that this formula is valid provided $|n-\nu|\notin\mathbb Z$, which is our case, and $\alpha\pm|n-\nu|+1\ne-1,-2,-3,...$.

We check the limit of the right-hand side of \eqref{P-} as $x\to 1^-$. From \cite[\S 14.8, 14.8.1]{nist} we have
\begin{equation}\label{nist2}
P_{\alpha}^{|n-1|}(x)\sim\frac{1}{\Gamma(1-|n-\nu|)}\left(\frac{2}{1-x}\right)^{\frac{|n-\nu|}{2}}\,,\ \ \ x\to 1^-,
\end{equation}
hence the left-hand side of \eqref{P-} is unbounded at $-1^+$ unless
$$
\alpha=|n-\nu|+k\,,\ \ \ k\geq 0.
$$
It remains to consider the cases $\alpha=\pm|n-\nu|+k-1$, $k\in\mathbb Z$, $k\leq -1$. We discard the case $\alpha=-|n-\nu|+k-1$: in fact, a corresponding eigenvalue would be $\mu=\alpha(\alpha+1)=(-|n-\nu|+k-1)(-|n-\nu|+k)=(|n-\nu|+\ell)(|n-\nu|+\ell+1)$ with $\ell=-k\in\mathbb Z$, $\ell\geq 1$, which we already know to be in the spectrum. Hence we are left with the case $\alpha=|n-\nu|+k-1$, $k\in\mathbb Z$, $k\leq -1$. In this case, formula \eqref{P-} is replaced by the following
\begin{multline}\label{P-2}
P_{|n-\nu|+k-1}^{-|n-\nu|}(-x)=\frac{\sin(\nu\pi)}{\sin(|n-\nu|\pi)}P^{-|n-\nu|}_{|n-\nu|+k-1}(x)\\
-\frac{\pi}{\sin(|n-\nu|\pi)(-k)!\Gamma(k+2|n-\nu|)}P^{|n-\nu|}_{\alpha}(x),
\end{multline}
and by \eqref{nist2} we see that the right-hand side is unbounded near $1^-$, hence the left-hand side is unbounded near $-1^+$.

\medskip

So far, we have found that for all $n\in\mathbb N$, the eigenvalues of \eqref{SL_sphere} are given by the collection
$$
\{(|n-\nu|+k)(|n-\nu|+k+1)\}_{k=0}^{\infty}.
$$

{\bf Step 2: compute multiplicities.} To count multiplicities, we assume for simplicity that $\nu\in(0,\frac 12]$.
We know  that the first eigenvalue is given by
$$
\mu_{00}=\nu(\nu+1).
$$
We have to distinguish the case $\nu\in(0,\frac 12)$ and $\nu=\frac 12$. 

\smallskip

If $\nu\in(0,\frac 12)$, then we consider the two families of eigenvalues corresponding to $n\geq 1$ and $n\leq 0$. In fact, we cannot find $k,j$ such that
$$
\mu_{nk}=\mu_{mj}
$$
for $n\geq 1$ and $m\leq 0$: we would have $n-\nu+k=-m+\nu+j$, which means $n+k+m-j=2\nu\notin\mathbb Z$. Now, we count the multiplicities for $n\geq 1$. The eigenvalue $(2-\nu)(1-\nu)$ occurs only for $n=1$ and it is the first eigenvalue $\mu_{10}$, hence it is simple. The eigenvalue $(2-\nu)(3-\nu)$ occurs only for $n=1,2$: it is the second eigenvalue $\mu_{11}$ for $n=1$ and the first eigenvale $\mu_{20}$ for $n=2$. Hence it has multiplicity $2$. In this way we see that the eigenvalue $(n-\nu)(n-\nu+1)$ has multiplicity $n$. Reasoning in the same way for $n\leq 0$, we see that the eigenvalue $(\nu-n)(\nu-n+1)$ has multiplicity $-n+1=|n|+1$. In particular, for $n=0$, the eigenvalue $\nu(\nu+1)$ has multiplicity $1$, and this is the first eigenvalue and it is simple:
$$
\mu_1(\mathbb S^2,A)=\nu(\nu+1).
$$

Next, we consider the case $\nu=\frac 12$. Clearly, all the discussion above holds, and we have two families of eigenvalues given by $(n-\nu)(n-\nu+1)$, $n\geq 1$ and multiplicity $n$, and $(\nu-m)(\nu-m+1)$ and multiplicity $-m+1=|m|+1$ for $m\leq 0$. However we note that the eigenvalues coincide if $n+m=1$. For example, for $n=1,m=0$ we have the eigenvalue $3/4$ with multiplicity $2$. For $n=2,m=-1$ we have the eigenvalue $15/4$ with multiplicity $4$. Altogether, we have that $(n-\nu)(n-\nu+1)$, $n\geq 1$, has multiplicity $2n$.

{\bf Completeness of the family of eigenfunctions.} To show that we have exhausted all the spectrum, we need to show that the family of functions $u_{nk}(r,\theta)=v_{nk}(r)e^{in\theta}$ is a complete system of eigenfunctions in $H^1_A(\mathbb S^2)$. This is a standard fact, and one can find a proof for more general manifolds of revolution around a point in  \cite[Appendix B]{CPS22}. 
\end{proof}

\subsection{Spectrum of the punctured hemisphere}
As a corollary of the above proposition we have that the spectrum of the hemisphere (with either Dirichlet or Neumann conditions) with pole at the center is equal to the spectrum of the whole sphere with antipodal poles (with different multiplicities).    

\medskip

Let $\mathbb S^2_+$ denote a hemisphere, and let $x_0$ be its center. Let $A^{(\nu)}_{x_0}=\nu d\theta$ the Aharonov-Bohm potential with pole $x_0$ and flux $\nu$. Then
\begin{equation*}
\{\mu_k(\mathbb S^2,A^{(\nu)}_{x_0})\}_{k=1}^\infty=\{\lambda_k(\mathbb{S}^2_+,A^{(\nu)}_{x_0})\}_{k=1}^\infty \cup \{\mu_k(\mathbb{S}^2_+,A^{(\nu)}_{x_0})\}_{k=1}^\infty
\end{equation*}
and the union is disjoint: namely, the spectrum on the whole sphere with opposite poles is the disjoint union of the Dirichlet and Neumann spectra for the potential $A^{(\nu)}_{x_0}$ on the hemisphere $\mathbb S^2_+$. The proof is standard, see e.g., \cite{bebe}. It is enough to observe that the radial parts of spherical eigenfunctions are either even or odd with respect to the equator. The restrictions of eigenfunction with odd radial part to $\mathbb S^2_+$ are Dirichlet eigenfunctions on $\mathbb S_+$, while the restrictions of the eigenfunctions with even radial part to $\mathbb S^2_+$ are Neumann eigenfunctions on $\mathbb S^2_+$. Vice-versa, odd/even reflecting Dirichlet/Neumann eigenfunctions of $\mathbb S^2_+$ give eigenfunctions on $\mathbb S^2$. Counting multiplicities we have:

\medskip

{\bf Dirichlet eigenvalues:} they are given by
  \begin{enumerate}
  \item If $\nu\in(0,\frac 12)$
\begin{align*}
            \{(n-\nu)(n-\nu+1)\}_{n\geq 1}& \text{ with multiplicity }{\frac{n-1}{2}}{\rm\  for\ } n\ {\rm odd\ and\ }{\frac{n}{2}}{\rm\  for\ } n\ {\rm even} \\
            \{(\abs n+\nu)(\abs n+\nu+1)\}_{n\leq 0}& \text{ with multiplicity }{\frac{|n|+1}{2}}{\rm\  for\ } n\ {\rm odd\ and\ }{\frac{|n|}{2}}{\rm\  for\ } n\ {\rm even} 
        \end{align*}
        \item if $\nu=\frac{1}{2}$  
        \begin{equation*}
            \{(n-\nu)(n-\nu+1)\}_{n\geq 1}\text{ with multiplicity }{n-1}{\rm\  for\ } n\ {\rm odd\ and\ }{n}{\rm\  for\ } n\ {\rm even} 
        \end{equation*}
    \end{enumerate}

\medskip
{\bf Neumann eigenvalues:} they are given by
  \begin{enumerate}
  \item If $\nu\in(0,\frac 12)$
\begin{align*}
            \{(n-\nu)(n-\nu+1)\}_{n\geq 1}& \text{ with multiplicity }{\frac{n+1}{2}}{\rm\  for\ } n\ {\rm odd\ and\ }{\frac{n}{2}}{\rm\  for\ } n\ {\rm even} \\
            \{(\abs n+\nu)(\abs n+\nu+1)\}_{n\leq 0}& \text{ with multiplicity }{\frac{|n|+1}{2}}{\rm\  for\ } n\ {\rm odd\ and\ }{\frac{|n|+2}{2}}{\rm\  for\ } n\ {\rm even} 
        \end{align*}
        \item if $\nu=\frac{1}{2}$  
        \begin{equation*}
            \{(n-\nu)(n-\nu+1)\}_{n\geq 1}\text{ with multiplicity }{n+1}{\rm\  for\ } n\ {\rm odd\ and\ }{n}{\rm\  for\ } n\ {\rm even} 
        \end{equation*}
    \end{enumerate}
    Note that the multiplicity of any Dirichlet/Neumann eigenvalue with $\nu=\frac 12$ is always an even number. For $\nu=0$ we find the multiplicities $n$ of the eigenvalue $n(n+1)$, $n\geq 0$, as Dirichlet eigenvalue of the Laplacian on $\mathbb S^2_+$, and the multiplicities $n+1$ of the eigenvalue $n(n+1)$, $n\geq 0$, as Neumann eigenvalue of the Laplacian on $\mathbb S^2_+$. By our convention, if the multiplicity is $0$ then the corresponding eigenvalue is not present in the spectrum.

\subsection{Limit of Dirichlet eigenvalues}
In the next lemma we prove that the first Dirichlet eigenvalue of a spherical cap centered at $x_0$ with potential $A^{(\nu)}_{x_0}$ converges to the first eigenvalue of the whole sphere $\mathbb S^2$ with the same potential when the radius of the cap tends to $\pi$.

\begin{lem}\label{lem_app_conv}
Let $D_R\subset\mathbb S^2$ be a geodesic disk of radius $R$ centered at $x_0$, and let $A^{(\nu)}_{x_0}=\nu d\theta$. Let $\lambda_1(D_R,A^{(\nu)}_{x_0})$ be the first eigenvalue of $\Delta_{A^{(\nu)}_{x_0}}$ with Dirichlet conditions on $\partial D_R$. Then
\begin{equation}
\lim_{R\rightarrow \pi} \lambda_1(D_R,A^{(\nu)}_{x_0})=\mu_1(\mathbb S^2,A^{(\nu)}_{x_0})
\end{equation}
\end{lem}
\begin{proof}
Through the proof we set $A\doteq A^{(\nu)}_{x_0}$. We can always assume $R>\frac\pi 2$ and $\nu\in(0,\frac 12]$.
Consider the function $\varphi$ defined by
\begin{equation}
        \varphi(r)=\begin{cases}
            1 \,\,&\text{ in }\,\, r\in[0,2R-\pi]\\
            \frac{r-R}{R-\pi} \,\,&\text{ in }\,\, r\in[2R-\pi,R]\\
            0 \,\,&\text{ in }\,\, r\in[R,\pi]
        \end{cases}
    \end{equation}
Let $u$ be the first eigenfunction on $\mathbb S^2$, associated with $\mu_1(\mathbb S^2,A)$. From Proposition \ref{prop:specsphere}, we know that it is radial and is given (up to constant multiples) by $u(r)=\sin^\nu(r)$. Set $C_R\doteq D_R\setminus D_{2R-\pi}$ the spherical annulus obtained by removing from $D_R$ a the disk of radius $2R-\pi$ centered at $x_0$. Now, we have that $u\varphi\in H^1_{0,A}(D_R)$ and from the min-max principle we have
\begin{equation}\label{eq:varlimit}
\lambda_1(D_R,A)\leq \frac{\int_{D_R}|d^A(u\varphi)|^2}{\int_{D_R}(u\varphi)^2}=\frac{\int_{D_{2R-\pi}}|d^Au|^2+\int_{C_R}|d^A (u\varphi)|^2}{\int_{D_{2R-\pi}}u^2+\int_{C_R}u^2\varphi^2}.
\end{equation} 
Everything is explicit, in fact $|d^A(u\varphi)|^2=|d(u\varphi)|^2+|A|^2u^2\varphi^2=u^2\varphi^2+\varphi^2|du|^2+\frac{\nu^2}{\sin^2(r)}u^2\varphi^2$. Standard computations show that
\begin{align*}
&\lim_{R\rightarrow \pi} \int_{C_R}|d^A (u\varphi)|^2=0\\
&\lim_{R\rightarrow \pi} \int_{C_R} u^2\varphi^2=0.
\end{align*}
Moreover, $\int_{D_{2R-\pi}}|d^A u|^2\rightarrow \int_{\mathbb{S}^2}|d^A u|^2$ and $\int_{D_{2R-\pi}}u^2\rightarrow \int_{\mathbb{S}^2}u^2$ when $R\rightarrow \pi$,
 and then from \eqref{eq:varlimit} we get that 
\begin{equation*}
\lambda_1(D_R,A)\leq \mu_1(\mathbb S^2,A)+o(1)\,\,\text{as}\,\, R\to \pi.
\end{equation*}
On the other hand, let $u_1$ be an eigenfunction on $D_R$ associated with $\lambda_1(D_R,A)$. Its extension by zero to $\mathbb S^2$ belongs to $H^1_{0,A}(\mathbb S^2)$, hence, from the min-max principle for $\mu_1(\mathbb S^2,A)$ we deduce that $\mu_1(\mathbb S^2,A)\leq \lambda_1(D_R,A)$. This concludes the proof.
\end{proof}

\section{The spectrum of flat cylinders}\label{appendixb}

Let $\mathcal C_M=\mathbb S^1\times[-M,M]$, with coordinates $(\theta,z)$, endowed with the flat metric $d\theta^2+dz^2$. Take the $1$-form $A=\nu d\theta$, which is closed, co-closed and has flux $\nu$ around any boundary component of $\mathcal C_M$. If we think of $\mathcal C_M$ as a cylinder in $\mathbb R^3$, the potential $1$-form can be expressed in cartesian coordinates as $\frac{\nu}{x^2+y^2}(-y dx+x dy)$.

In order to compute $\mu_1(\mathcal C_M,A)$, we write the magnetic Laplacian $\Delta_A$ in cylindrical coordinates. We obtain 
\begin{align*}
    \Delta_A u&=\Delta u+|A|^2u^2+2i\langle A,\nabla u\rangle\\
    &=-u_{zz}-u_{\theta\theta}+\nu^2u+2i\nu u_{\theta}.
\end{align*}
By separation of variables, we can write the eigenfunctions in the form $u(\theta,z)=e^{in\theta}v_n(z)$, $n\in\mathbb Z$ (see e.g., \cite{PSannuli} for more details). Plugging this ansatz in the eigenvalue equation, and denoting by $\mu_n$ a corresponding eigenvalue, we get
\begin{equation*}
    -v_n''(z)=\left( \mu_n-\left| n- \nu\right |^2 \right )v_n(z)\,\,\,\,\, \forall z\in (-M,M),
\end{equation*}
and the boundary conditions, due to the fact that $\langle A,N\rangle=0$ on $\partial\mathcal C_M$, just read $v_n'(z)=0$ at $z=-M$ and $z=M$.
This is just a shifted Neumann problem on the segment $(-M,M)$, hence the spectrum is given by the collection
\begin{equation*}
\left\{|n-\nu|^2+\frac{k^2\pi^2}{4M^2}\right\}_{k\in\mathbb N,n\in\mathbb Z}
\end{equation*}
and the first eigenvalue is given by
\begin{equation*}
   \mu_1(\mathcal C_M,A)=\min_{n\in\mathbb Z}  \left| n- \nu\right |^2.
\end{equation*}
By gauge invariance, we may always assume $\nu\in(0,\frac 12]$, so  that the minimum is attained at $n=0$ and we can conclude that
\begin{equation*}
   \mu_1(\mathcal C_M,A)=\nu^2.
\end{equation*}
independently on $M$. The corresponding eigenspace is spanned by  the constant function $u_1=1$.

The fact that the first eigenfunction of any flat annulus has constant modulus will imply an upper bound of isoperimetric type on any Riemannian cylinder, similarly to what happens to the magnetic Steklov problem in cylinders, see \cite{PSannuli}. This fundamental fact is used in the proof of Theorem \ref{main_annulus}.

\begin{proof}[Proof of Theorem \ref{main_annulus}]
Let $\mathcal C$ be a Riemannian annulus, namely, a differentiable surface diffeomorphic to $\mathbb S^1\times[0,1]$ endowed with a Riemannian metric. It is known that any Riemanniann annulus $\mathcal C$ is conformally equivalent to the flat cylinder $\mathcal C_M$ for a unique $M>0$, which is called the {\it conformal modulus} of $\mathcal C$. In other words, there exists a conformal diffeomorphism $\Phi:\mathcal C\to\mathcal C_M$. Consider on $\mathcal C_M$ the potential $1$-form $A=\nu d\theta$, which is closed and has flux $\nu$. Consider on $\mathcal C$ the potential $1$-form $\Phi^{\star}A$, which is again closed, and with flux $\nu$. It is unique up to exact forms. By gauge invariance, the magnetic spectrum of $\mathcal C$ depends only on $\nu$, see e.g., \cite[\S4]{CPS22} or \cite{PSannuli} for more details. Let $u_1=1$. By the min-max principle we have

\begin{equation}
\mu_1(\mathcal C,\Phi^{\star}A)\leq\frac{\int_{\mathcal C}|d^{\Phi^{\star}A}u_1|^2}{\int_{\mathcal C}|u_1|^2}=\frac{\int_{\mathcal C}|\Phi^{\star}A|^2}{|\mathcal C|}=\frac{\int_{\mathcal C_M}|A|^2}{|\mathcal C|}=\frac{|\mathcal C_M|}{|\mathcal C|}\nu^2.
\end{equation}
That is
$$
|\mathcal C|\mu_1(\Sigma,\Phi^{\star}A)\leq|\mathcal C_M|\mu_1(\mathcal C_M,A).
$$
As already mentioned, the Aharonov-Bohm spectrum of a Riemannian cylinder depends only on the flux $\nu$, hence we may write $\mu_1(\mathcal C,\nu)\doteq\mu_1(\mathcal C,\Phi^{\star}A)$ and $\mu_1(\mathcal C_M,\nu)\doteq\mu_1(\mathcal C_M,A)$.
If equality holds, then $u_1=1$ is an eigenfunction of $\Delta_{\Phi^{\star}A}$ on $\mathcal C$, which means that $|\Phi^{\star}A|^2=\nu^2$. This means that $\mathcal C$ supports a closed $1$-form of constant length. But since since the metric on $\mathcal C$ is conformal to the flat metric, we have that $|\Phi^{\star}A|^2=\rho^2|A|^2=\rho^2$ for some conformal factor $\rho$, and this implies that $\rho$ is constant. Hence $\mathcal C$ is homothetic to $\mathcal C_M$.
\end{proof}

\bibliography{bibliography.bib}
\bibliographystyle{abbrv}
\end{document}